\documentclass{amsart}
\usepackage[all]{xypic}
\usepackage[latin1]{inputenc} 
\usepackage[T1]{fontenc}
\usepackage{ae,aecompl}
\usepackage{verbatim}
\usepackage{amsfonts}
\usepackage[pdftex]{graphicx}
\newcommand{\Z}{\mathbb{Z}}

\newcommand{\R}{\mathbb{R}}
\newcommand{\C}{\mathbb{C}}
\newcommand{\V}{\mathbf{V}}
\newcommand{\A}{\mathbf{A}}
\newcommand{\gm}{\displaystyle}
\newcommand{\one}{\mathbf{1}}
\newcommand{\Png}[2]{
\centerline{\includegraphics[height=#1]{#2}}
}
\newcommand{\png}[4]{
\hspace{#1pt}\raisebox{#2pt}{\includegraphics[height=#3pt]{#4}}\hspace{#1pt}
}
\newtheorem{thm}{Theorem}[section]
\newtheorem{pro}[thm]{Proposition}

\newtheorem{lem}[thm]{Lemma}

\theoremstyle{remark}
\newtheorem{rem}[thm]{Remark}
\newtheorem{exo}[thm]{Exercise}
\newtheorem{exa}[thm]{Example}

\title {An oriented model for Khovanov homology
}
\author{Christian Blanchet}

\begin{document}
\begin{abstract}

We give an alternative presentation of Khovanov homology of links.
 The original construction rests on the Kauffman bracket model for
 the Jones polynomial, and the generators for the complex are  enhanced Kauffman
 states. Here we use an oriented $sl(2)$ state model allowing a natural definition
 of the boundary operator as twisted action of morphisms belonging to a TQFT
 for trivalent graphs and surfaces. Functoriality in original Khovanov homology holds up to sign.
 Variants of Khovanov homology fixing functoriality were obtained by Clark-Morrison-Walker \cite{CMW}
and also by Carmen Caprau \cite{CC}.
 Our construction is similar to those variants. 
 Here we work over integers, while the previous constructions were over gaussian integers, and produce the TQFT
by a universal construction. We consider diagrams in the oriented plane. Our functoriality results include  that for a fixed link the homology isomorphism associated with a sequence of Reidemeister moves between two diagrams is canonical.
%We were able to generalize our method and obtain a categorification
%of $sl(N)$ quantum invariant; this will be written elsewhere.
\end{abstract}
\maketitle
\section{Trivalent TQFT}\label{s1}
\subsection{Frobenius algebra}
TQFTs for oriented surfaces are in one to one correspondence with commutative Frobenius algebras (also called symmetric
algebras)
\cite{Ko}. We consider here the Frobenius algebra $\A=\Z[X]/X^2\approx H^*(\C P^1)$, 
and we denote by  $V_\A$ the associated TQFT.
 The unit element in $\mathbf{A}$ is denoted by $\mathbf{1}$; the coalgebra structure $(\Delta,\epsilon)$ on 
 $\mathbf{A}$ is defined by 
 $$\epsilon(X)=1\ ,\ \epsilon(\mathbf{1})=0\ ;$$
 $$\Delta(X)=X\otimes X\ ,\ \Delta(\mathbf{1})=\mathbf{1}\otimes X + X\otimes\mathbf{1}\ .$$
The invariant of a closed surface is given below.
$$V_\A(S^1\times S^1)=\mathrm{rank}(\mathbf{A})=2\ ,$$
$$V_\A(\Sigma_g)=0 \text{ for a closed surface $\Sigma_g$ with genus $g\neq 1$.}$$
The TQFT is extended to surfaces with points. The neighbourhood of a point
represents the element $X$ in the algebra associated with the oriented circle
\mbox{$V_\A(S^1)=\mathbf{A}$}.
For a genus $g$ closed surface with $k$ points the invariant is zero excepted
\begin{verse}
for $(g,k)=(1,0)$ where the value is $2$, and\\
for $(g,k)=(0,1)$ where the value is $1$.
\end{verse}

\subsection{The universal construction}\label{universal}
One can reconstruct the above TQFT from the invariant of closed surfaces with points.
The TQFT module of an oriented curve $\gamma$ is generated over $\mathbb{Z}$ by surfaces with points
whose boundary is identified with $\gamma$. Relations are given by the kernel
of the bilinear form defined by gluing.
A key point in proving that the functor $V_A$ defined this way is indeed
a TQFT is the surgery formula in Figure \ref{surgery}.\\[10pt]
\begin{figure}
\centerline{\raisebox{5mm}{\LARGE $V_\A(\ $}\includegraphics[height=12mm]{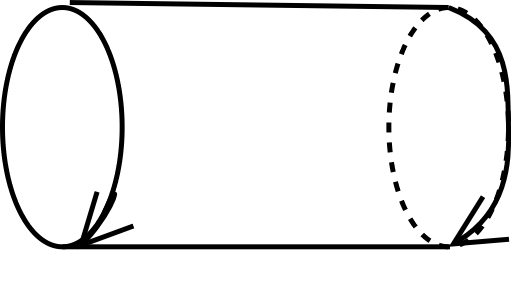}\raisebox{5mm}{\LARGE $\ )$}
\raisebox{5mm}{\LARGE $\ =\ $}\raisebox{5mm}{\LARGE $V_\A(\ $}\includegraphics[height=12mm]{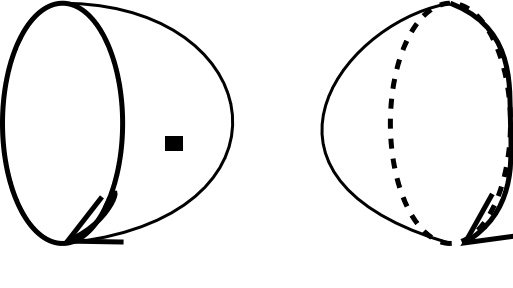}\raisebox{5mm}{\LARGE $\ )$}\raisebox{5mm}{\LARGE $\ +\ $}
\raisebox{5mm}{\LARGE $V_\A(\ $}\includegraphics[height=12mm]{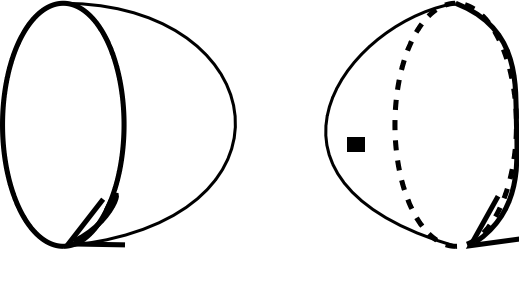}\raisebox{5mm}{\LARGE $\ )$}
}
\caption{\label{surgery} Surgery formula}
\end{figure}
Here cobordisms are depicted from left to right. The graphical identity can be  written
$$\mathrm{Id}_{\mathbf{A}}=\epsilon(X\times\cdot)\mathbf{1}+\epsilon(\cdot)X$$
\subsection{Graded TQFT}
We define a grading on $\mathbf{A}=\Z[X]/X^2$, by $\deg(\mathbf{1})=1$, 
$\deg({X})=-1$.
 The $q$-dimension (or Poincar\'e polynomial) of the TQFT module associated with
a $k$ components curve is $(q+q^{-1})^k$.
 The TQFT functor is graded. For a cobordism  $\Sigma$ between $\gamma$ and
$\gamma'$, 
the linear map $$V_\A(\Sigma):V_\A(\gamma)\rightarrow V_\A(\gamma') $$ has degree $\chi(\Sigma) -2\sharp\mathrm{pts}$. Here $\chi(\Sigma)$ is the Euler characteristic, and $\sharp\mathrm{pts}$ is the number of points.

\subsection{Trivalent category}
We will extend the TQFT over the cobordism category whose objects are trivalent graphs and whose morphisms are trivalent surfaces.
 Here a 
 {\em trivalent graph} is an oriented graph with edges labelled with $1$ or $2$,
and $3$-valent vertices where the flow condition is respected. For each trivalent vertex
an order on the $2$ (germs of) edges  labelled with $1$ is fixed. For a planar graph we use plane orientation and fix the order according to the rule depicted in figure \ref{trivalent}; the labels of the edges are obviously encoded in the arrows.
\begin{figure}
\centerline{\includegraphics[height=25mm]{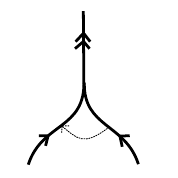}\hspace{1cm}
\includegraphics[height=25mm]{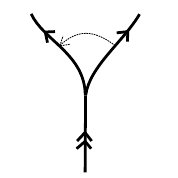}
}\caption{\label{trivalent}Trivalent vertices}
\end{figure}

A closed trivalent surface is a $2$-complex whose
 regular faces are oriented and labelled with $1$ or $2$.
 The singular locus is a curve called the binding; each component of the binding has a neighborhood which is a trivalent vertex times $S^1$, i.e. there are two  $1$-labelled pages inducing the same orientation and one   $2$-labelled page inducing the opposite orientation.
 For each  component of the binding,
an order on the two $1$-labelled   pages is fixed. A $1$-labelled face may have points on it.

Cobordisms are obtained by cutting in a generic way. They are considered up to the usual equivalence
of oriented homeomorphism rel. boundary.

\subsection{TQFT on the trivalent category}
The following general procedure constructs a functor (hopefully a TQFT functor) on the trivalent cobordism category. We first define an invariant of closed trivalent surfaces,
and extend it into a functor on the trivalent category via the universal construction
introduced in \cite{BHMV} and sketched above (\ref{universal}) for surfaces.

Suppose that we are given Frobenius algebras $A$, $B$ and $C$ over a ring $\mathbf{k}$, with corresponding TQFT functors denoted by $V_A$, $V_B$ and $V_C$.
 Let $\Sigma$ be a closed trivalent surface, and $\dot \Sigma=\Sigma_1\amalg \Sigma_2$ be the surface cut along the binding, decomposed according to the label of the faces.
 Let $m$ be the number of components of the binding of $\Sigma$. The boundary of $\Sigma_1$ has $2m$ oriented components $C_i^-$ and $C_i^+$, $1\leq i\leq m$, and the boundary of $\Sigma_2$ has $m$ components $C_i^2$.
 Here the $\mp$ is fixed with respect to the ordering of the $1$-labelled pages.
The TQFT functors $V_A$ and $V_B$ associate to $\Sigma_1$ and $\Sigma_2$  vectors
$$V_A(\Sigma_1)\in \bigotimes_{i=1}^m \left(V_A(C_i^-)\otimes V_A(C_i^+) \right)\ \cong \left({A}\otimes {A}\right)^{\otimes m}\ ,$$
$$V_B(\Sigma_2)\in \bigotimes_{i=1}^m V_B(C_i^2)\ \cong B^{\otimes m}\ .$$
 Now suppose that we are given maps $f=A\otimes A\rightarrow C$,
$g: B\rightarrow C$, then we
define the invariant $\V(\Sigma)$ by the formula
$$\mathbf{V}(\Sigma)= (\epsilon_C)^{\otimes m}\left(f^{\otimes m}(V_A(\Sigma_1))\times g^{\otimes m}(V_B(\Sigma_2)\right)\in \mathbf{k}^{\otimes m}=\mathbf{k}\ .$$
Here $\epsilon_C: C\rightarrow \mathbf{k}$ is the trace on the Frobenius algebra $C$; the product is computed in $C^{\otimes m}$.

From now on, we  use the Frobenius algebras over $\Z$: $\A=\Z[X]/X^2\approx H^*(\C P^1)$, $C=\A$, and $B=\Z$ with non standard trace $\epsilon_C(n)=-n$. The structural map $f$ is defined by $f(x\otimes y)= x\overline y$, where $\overline{a+bX}=a-bX$ ($a,b\in\Z$),
and $g : C=\Z\rightarrow \A=B$ is the unit map.  
\begin{exa}
 Let us consider the trivalent surface which is a sphere together with a $2$-labelled meridional disk, and whose $1$-labelled half-spheres are ordered north-south. The associated value
\begin{verse}
 is $0$ if there is no point,\\
is $1$ if there is one point which is on north half-sphere,\\
is $-1$ if there is one point which is on south half-sphere,\\
is $0$ if there is more than one point.
\end{verse}
\end{exa}
The universal construction extends the invariant $\V$ to a functor on the trivalent cobordism category.
 The following proposition shows that the functor $\V$ is an extension of the TQFT functor $V_A$.

\begin{pro}
a) We have a natural transformation from $V_A$ to
$\V$, i.e. for an oriented curve $\gamma$, we have a TQFT module $V_\A(\gamma)$, a module $\V(\gamma)$, and a natural map
$$i_\gamma : V_\A(\gamma)\rightarrow \V(\gamma)\ .$$
{\rm Here we label all the components of $\gamma$ with $1$, and consider $\gamma$ as an object in the trivalent category.}\\
b) For any curve $\gamma$, the natural map $i_\gamma$ is an isomorphism. 
\end{pro}
\begin{proof}
The surgery formula in Figure \ref{surgery}
holds for surgery on a $1$-labelled face of a trivalent surface. Using this formula along each component of the curve $\gamma$, we see that any trivalent surface with boundary $\gamma$ representing a generator of $\V(\gamma)$ can be written as a linear combination
of disks (may be with points), and that any linear combination representing a relation in $V_A(\gamma)$
also represents a relation in $\V(\gamma)$. This proves existence and surjectivity of $i_\gamma$. Injectivity of $i_\gamma$ and naturality follow from the definitions in the universal construction.
\end{proof}
The extended  functor $\V$ is still graded. The formula for a cobordism  $\Sigma$ is
$$\deg(\Sigma)=\chi_1(\Sigma) -2\sharp\mathrm{pts}\ .$$
Here $\chi_1(\Sigma)$ is the Euler characteristic of the $1$-labelled subsurface, e.g. the saddle with $2$-labelled membrane in figure \ref{saddlem} has grading $-1$. Here $1$-labelled faces are depicted in light grey and the $2$-labelled
half-disc is black.
\begin{figure}
\centerline{\includegraphics[height=40mm]{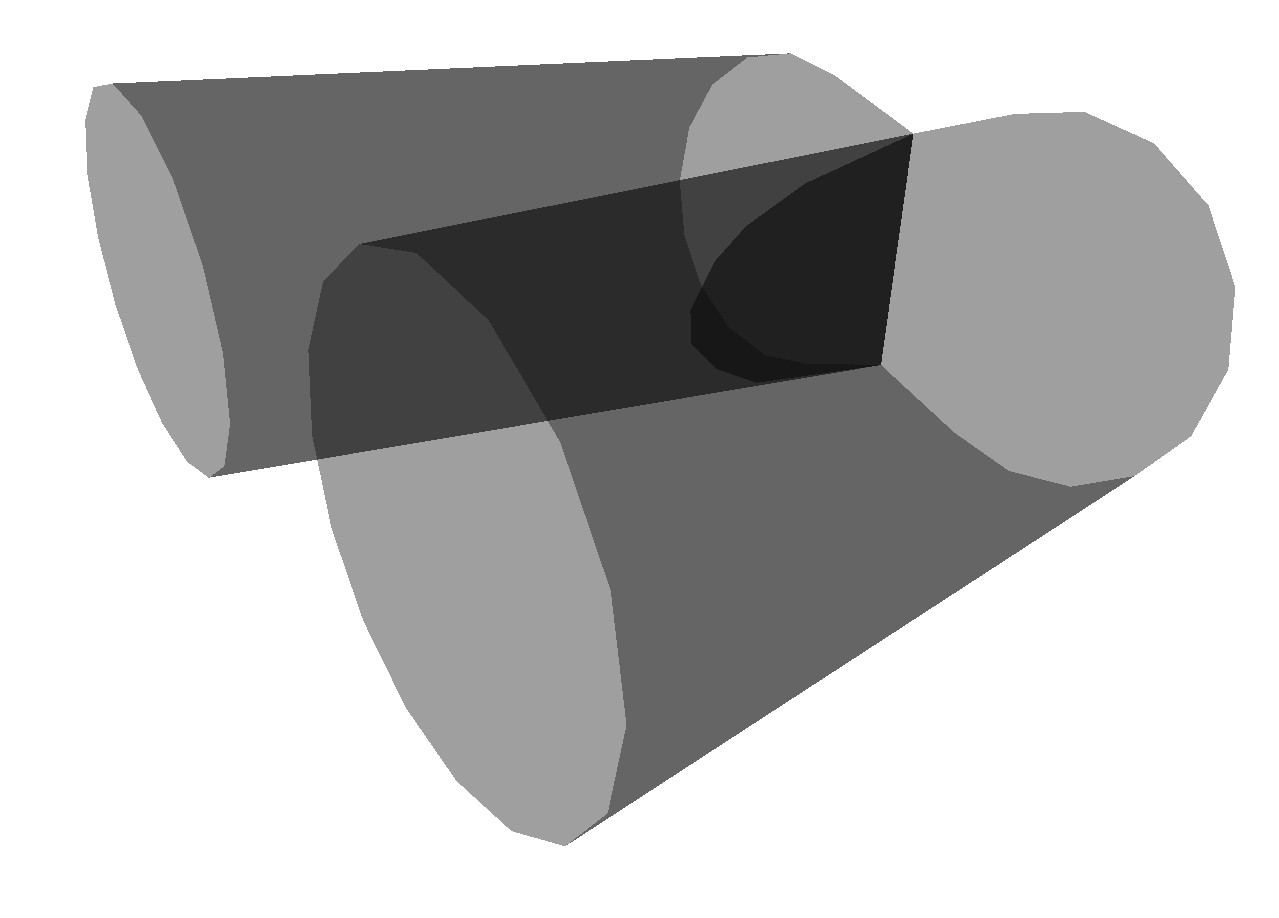}}
 \caption{\label{saddlem} Trivalent surface with grading $-1$}
\end{figure}

The lemma below gives some examples of computation with the extended functor $\mathbf{V}$.
These results will be useful in the subsequent categorification procedure. In the pictures, the order on the germs of $1$-labelled edges is fixed by the following  plane convention (Figure \ref{trivalent}): the first $1$-labelled edge is on the right
of the oriented $2$-labelled adjacent edge. Proofs are left as exercise.
\begin{lem} \label{lemabcd}
a) If $\Sigma'$ is obtained from $\Sigma$ by moving a point across a component of the  binding then
$\V(\Sigma')=-\V(\Sigma)$.\\
b) The bubble relations in Figure \ref{bubbles} hold.\\
c) The band moves relations in Figure \ref{band} hold.\\
d) The tube relations in Figure \ref{tube} hold (the sign depends on the order of the $1$-labelled pages at each binding).
\end{lem}
%\marginpar{dessins}
In the pictures the $2$-labelled faces are depicted in black, the $1$-labelled faces are depicted in grey.
The small arc indicates the order around the binding.
\begin{figure}
\centerline{\includegraphics[height=30mm]{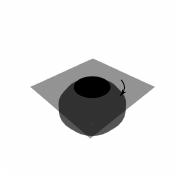}
\raisebox{12mm}{$\displaystyle=$}
\includegraphics[height=30mm]{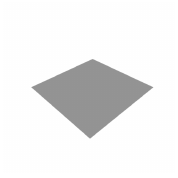}
\raisebox{12mm}{$\displaystyle=\ -$}
\includegraphics[height=30mm]{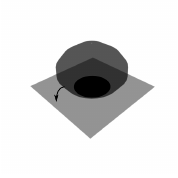}
}
\centerline{\includegraphics[height=30mm]{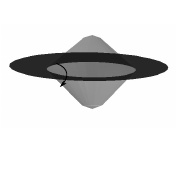}
\raisebox{17mm}{$\displaystyle=\ 0$}}
\centerline{\includegraphics[height=30mm]{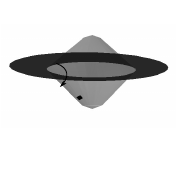}
\raisebox{17mm}{$\displaystyle=$}
\includegraphics[height=30mm]{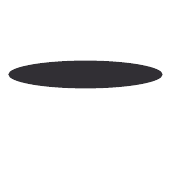}
\raisebox{17mm}{$\displaystyle=\ -$}
\includegraphics[height=30mm]{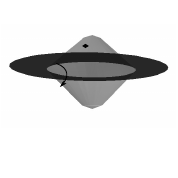}
}
\caption{\label{bubbles} Bubble relations}
\end{figure}
\begin{figure}
\centerline{\includegraphics[height=40mm]{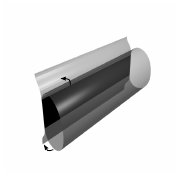}
 \raisebox{16mm}{$\displaystyle=\ -$}
 \includegraphics[height=40mm]{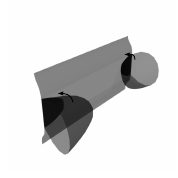}
}
\centerline{\includegraphics[height=40mm]{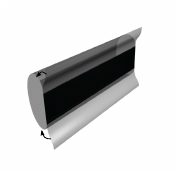}
\raisebox{16mm}{$\displaystyle=$}
\includegraphics[height=40mm]{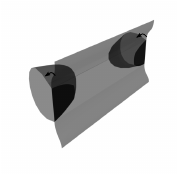}
 }
\caption{\label{band} Band relations}
\end{figure}

\begin{figure}
\centerline{\includegraphics[height=40mm]{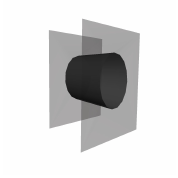}
 \raisebox{16mm}{$\displaystyle=\ \pm$}
 \includegraphics[height=40mm]{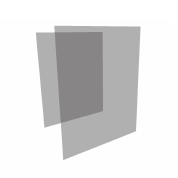}
}
\centerline{\includegraphics[height=40mm]{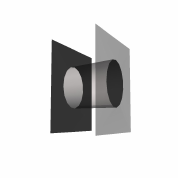}
\raisebox{16mm}{$\displaystyle=\ \pm$}
\includegraphics[height=40mm]{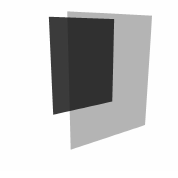}
 }
\caption{\label{tube} Tube relations}
\end{figure}

\section{Categorification of the $sl(2)$ invariant of planar graphs}
We consider here trivalent planar graphs whose edges are smooth; each edge has a label equal to $1$ or $2$.
 In each trivalent vertex, the flow is conserved, and the   tangent vectors are coherent.
  Loops with label $1$  are accepted. The admissible vertices are depicted in Figure
\ref{trivalent}. 
  In the representation theoretic setting, $1$-labelled edges correspond to the standard representation of $sl(2)$, and $2$-labelled edges correspond to its determinant (isomorphic to the trivial representation).

An enhancement $\epsilon$ for such a graph is a map from the set of $1$-labelled edges to $\{-1,1\}$ required to have distinct values for the two edges adjacent to a trivalent vertex.
 To each trivalent vertex $v$ we associate a weight 
$ \mathcal{W}(v) = q^{\pm\frac{1}{2}}$. Here $q$ is an indeterminate, and the sign is given
by the state of the right handed edge.
The $sl(2)$ invariant of such a graph $G$ is given by
$$\langle G\rangle=\prod_{\mathrm{vertices\ } v}\mathcal{W}(v) \ q^{\sum_{a}\epsilon(a)\mathrm{rot}(a)}\ .$$
The sum is over all $1$-labelled edges $a$, and $\mathrm{rot}(a)$ is the variation of the tangent vector along the edge, normalized so that it gives the Whitney degree (signed number of rotation) for a closed curve.

 The invariant $\langle G\rangle$ is easily seen to be equal to $(q+q^{-1})^{\sharp G_1}$ where $\sharp G_1$ is the number of components of the curve composed with the $1$-labelled edges. Its interest is that it allows to give a state model for the Jones polynomial
similar to the Kauffman bracket state model, but taking into account the orientation.

 We associate to such a graph the graded module
$\V(G)=\oplus_k \V_k(G)$. For any graph $G$, the module $\V(G)$ admits a finite set of generators which can be obtained
by first pairing the trivalent vertices with singular arcs and then gluing discs, may be with points on the $1$-labelled ones.
 The module itself can then be computed using the pairing.
\begin{exo}
Let $G_1$, $G_2$, $G3$, $G_4$ be the graphs depicted in Figure \ref{graphs}. Show that
$$\V(G_1)\approx \A\ , \V(G_2)\approx \A^{\otimes 2}\ ,\ \V(G_3)\approx \A^{\otimes 2}\ ,\ \V(G_4)\approx \A\ .$$
\end{exo}
\begin{figure}
\centerline{\includegraphics[height=40mm]{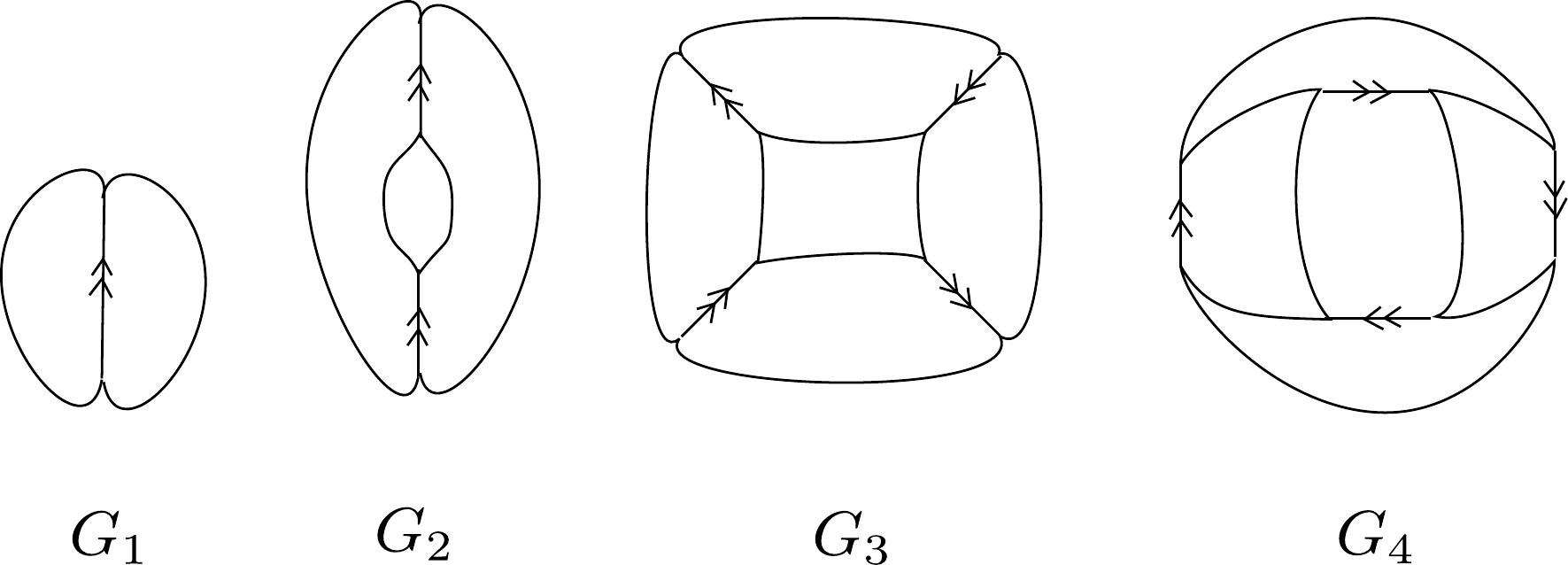}
 }
\caption{\label{graphs}}
\end{figure}

\begin{pro}
$\V(G)$ is a free abelian group whose $q$-dimension is equal to the invariant
$\langle G\rangle$:
$$\langle G\rangle=\sum_k \ q^k\ \mathrm{rank}( \V_k(G))\ .$$
\end{pro}

\begin{rem} 
We understand this theorem as a categorification of the invariant $\langle G\rangle$.
 Indeed, the functor $\V$ associates to a graph $G$ a graded abelian group which can be interpreted as (co)homology concentrated in (co)homological degree zero. The purpose of the next section is to extend this categorification to link diagrams.
\end{rem} 

\begin{proof}
 It is enough to prove the formula for a connected graph.
We proceed by induction on the number of $2$-labelled edges.
 If this number is $0$, we get a loop whose value is $q+q^{-1}$.
 By an Euler characteristic argument a trivalent graph will have at least one face which is either a bigon or a square. Lemmas \ref{bigons} and \ref{squares} below shows that the computation reduces to graphs with less $2$-labelled edges.
\end{proof}
\begin{lem}[Bigons]\label{bigons}
a) $\gm \V(\raisebox{-4mm}{\includegraphics[height=12mm]{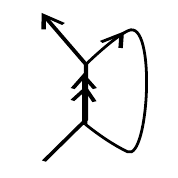}}\ )\simeq \V(\raisebox{-4mm}{\includegraphics[height=12mm]{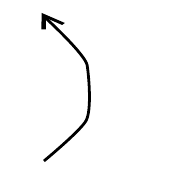}})$.\\
b) $\gm \V(\raisebox{-4mm}{\includegraphics[height=12mm]{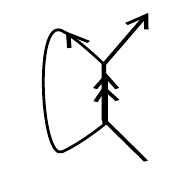}}\ )\simeq \V(\raisebox{-4mm}{\includegraphics[height=12mm]{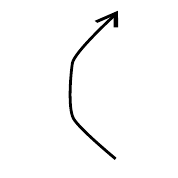}})$.\\
c) $\gm \V(\raisebox{-7mm}{\includegraphics[height=16mm]{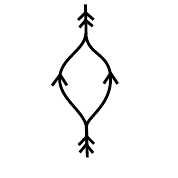}} )\simeq\V(\raisebox{-7mm}{\hspace{-10pt}\includegraphics[height=16mm]{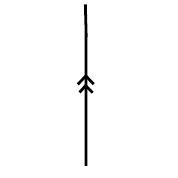}}\hspace{-10pt})\{-1\}\oplus 
\V(\raisebox{-7mm}{\hspace{-10pt}\includegraphics[height=16mm]{G/double.pdf}}\hspace{-10pt})\{1\} $.
\end{lem}
Here the bracket in right hand side of c) indicates a shift in the grading.
\begin{proof}
 a) and b) are deduced from the band relations in Lemma \ref{lemabcd}.c.
Indeed the cobordisms in the right hand side of these relations can be decomposed by cutting in the middle. The induced TQFT maps give the needed isomorphisms.

Lemma \ref{bigon} below, whose proof is left to the reader, decomposes identity
of the module on the left hand side of c) into two orthogonal idempotents whence the direct sum decomposition. Note the shift given by the degree of the cobordisms inducing the projection on each summand.
 \end{proof}
\begin{lem}\label{bigon}
 The relations in figure \ref{fbigon} and \ref{4t} holds.
\end{lem}
\begin{figure}
\centerline{\includegraphics[height=30mm]{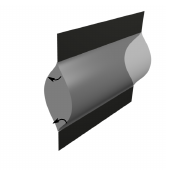}
\raisebox{12mm}{$\displaystyle=$}
\includegraphics[height=30mm]{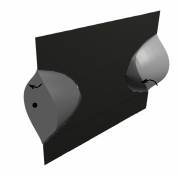}
\raisebox{12mm}{$\displaystyle\ -$}
\includegraphics[height=30mm]{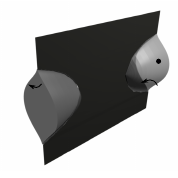}
}
\caption{\label{fbigon} Bigon relation}
\end{figure}
\begin{figure}
\centerline{\includegraphics[height=30mm]{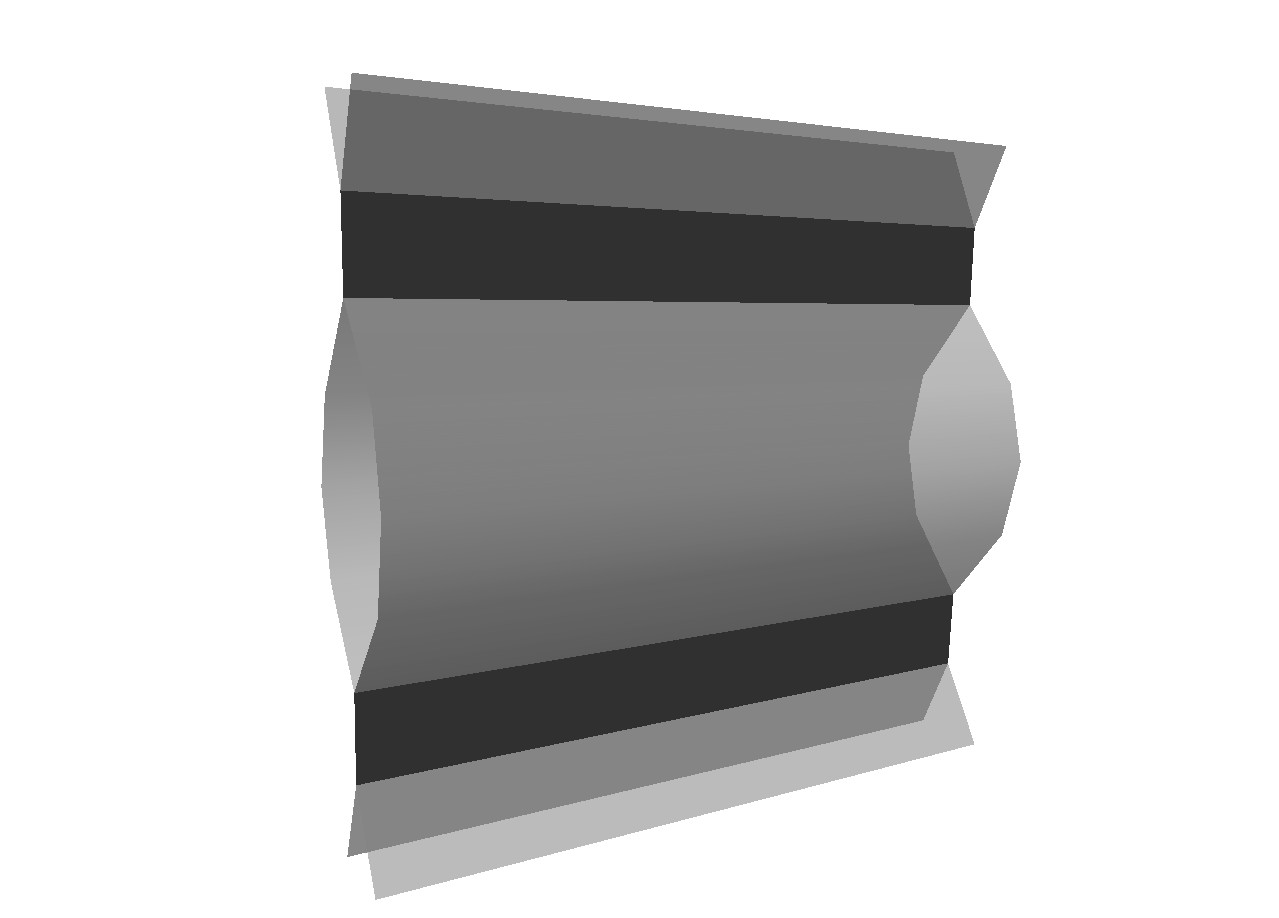}
\raisebox{12mm}{$\displaystyle=$}
\includegraphics[height=30mm]{G/bigon2.pdf}
\raisebox{12mm}{$\displaystyle\ -$}
\includegraphics[height=30mm]{G/bigon3.pdf}
}
\caption{\label{4t} $4$ terms relation}
\end{figure}
\begin{lem}[Squares]\label{squares}
a) $\gm \V(\raisebox{-5mm}{\includegraphics[height=12mm]{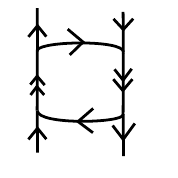}} )\simeq\V(\raisebox{-5mm}{\includegraphics[height=12mm]{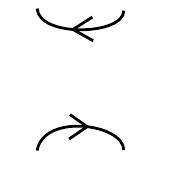}})
%\oplus \V(\raisebox{-4mm}{\includegraphics[height=12mm]{G/square1s2.pdf}}) 
$.\\
b) $\gm \V(\raisebox{-5mm}{\includegraphics[height=12mm]{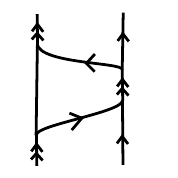}} )\simeq\V(\raisebox{-5mm}{\includegraphics[height=12mm]{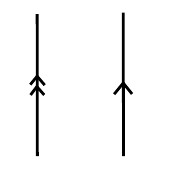}})$.
\end{lem}
\begin{proof}
 Isomorphisms in a) are depicted in figure \ref{squaresf}. Both compositions give identity up to sign. One can be seen using relations in Lemma \ref{lemabcd}.
Isomorphisms in b) are described in figure \ref{squaresg}.
\end{proof}
\begin{figure}
\centerline{\includegraphics[height=32mm]{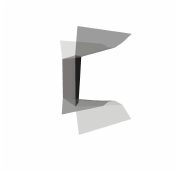}
\raisebox{3mm}{\includegraphics[height=30mm]{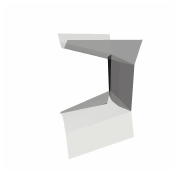}}
}
\caption{\label{squaresf} Isomorphisms in Lemma \ref{squares}a)}
\end{figure}
\begin{figure}
\centerline{\includegraphics[height=33mm]{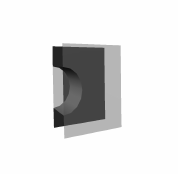}
\includegraphics[height=30mm]{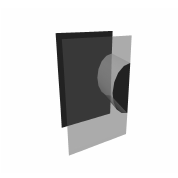}
}
\caption{\label{squaresg} Isomorphisms in Lemma \ref{squares}b)}
\end{figure}

\section{Khovanov homology}
\subsection{Jones polynomial via planar graphs}
The formulas below extend the preceeding $sl(2)$ invariant
of planar trivalent graphs to an invariant of link diagrams.
\begin{center}
\vspace{5pt}
\raisebox{5mm}{$\langle$}
\includegraphics[height=12mm]{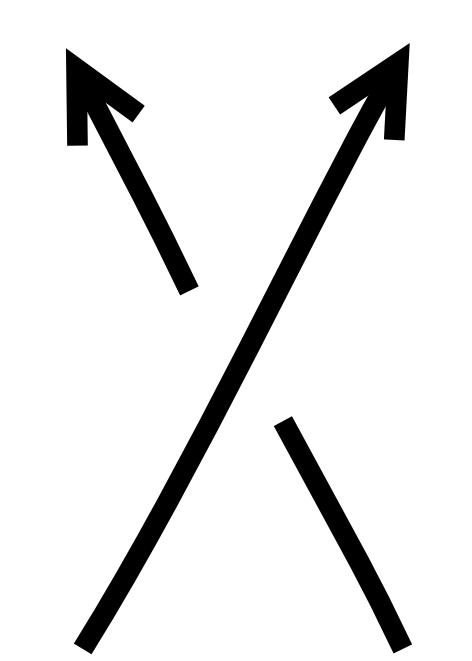}\raisebox{5mm}{$\rangle\ =\ q^{-1} \ \langle$}\includegraphics[height=12mm]{G/c0}\raisebox{5mm}{$\rangle\ -\ q^{-2} \langle$}\includegraphics[height=12mm]{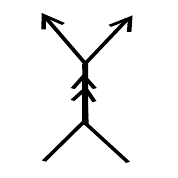}\raisebox{5mm}{$\rangle$}\\
\vspace{5pt}
\raisebox{5mm}{$\langle$}
\includegraphics[height=12mm]{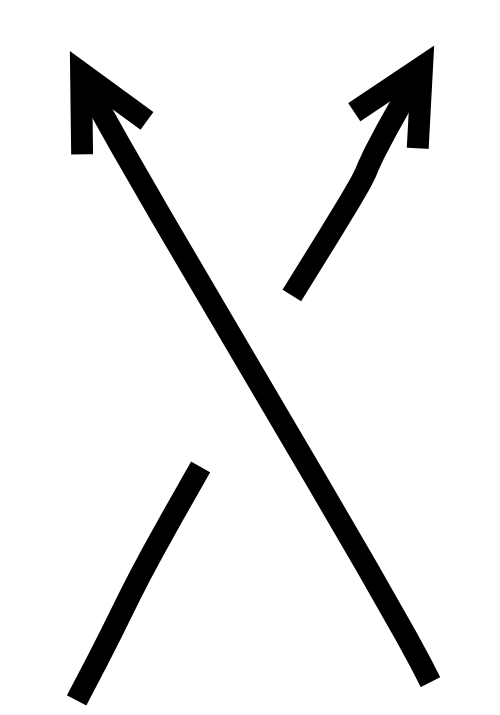}\raisebox{5mm}{$\rangle\ =\ q
\ \langle$}\includegraphics[height=12mm]{G/c0}\raisebox{5mm}{$\rangle\ -\ q^2 \langle$}\includegraphics[height=12mm]{G/arete_double.pdf}\raisebox{5mm}{$\rangle$}
\end{center}
The normalisation for the empty link is $1$, and we have the following skein relation
\begin{equation}\label{skeinJones}\text{
{$q^{2}\langle$}
\raisebox{-5mm}{\includegraphics[height=12mm]{G/c+}}{$\rangle\ 
-\ q^{-2} \langle$}\raisebox{-5mm}{\includegraphics[height=12mm]{G/c-}}{
$\rangle
=\ (q-q^{-1})\ \langle$}\raisebox{-5mm}{\includegraphics[height=12mm]{G/c0}}{$\rangle$}
}\end{equation}
Up to normalisation, we recognize the Jones polynomial with the change of variable $q=-t^{-\frac{1}{2}}$. A global state sum formula for a link represented by a diagram $D$ is given below. Note that it is quite easy to show that this formula is invariant under Reidemeister
move; this is a slight variant of the Kauffman bracket construction.

We give a global state sum formula for a diagram $D$. A state $s$ of $D$ associates to a positive (resp. negative) crossing either $0$ or $1$
(resp. $-1$ or $0$). For a state $s$,
$D_s$ is the planar trivalent graph, defined by the rule:\\
\centerline{if $s(c)=0$, then $c$ is replaced by 
\raisebox{-4mm}{\includegraphics[height=12mm]{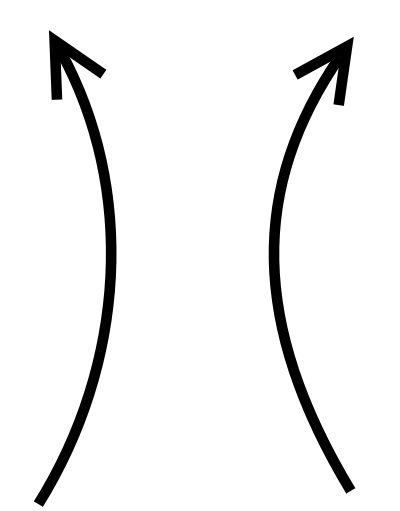}}}
\centerline{if $|s(c)|=1$, then $c$  is replaced by 
\raisebox{-4mm}{\includegraphics[height=12mm]{G/arete_double.pdf}}}

One has
\begin{equation}
 \label{globalState}
\langle D\rangle=\sum_s q^{-(w(D)+s(D))}\langle D_s\rangle
\end{equation}
Here $w(D)=\sum_c\mathrm{sign}(c)$, and $s(D)=\sum_cs(c)$.
\subsection{Khovanov complex}
We consider a link diagram $D$.
For a state  $s$, we define the trivalent graph $D_s$ according
to the local rules described just above.  We use the notation $d_s$ for $\sum |s(c)|$, and
$\Delta_s$ for the free abelian group generated by crossings  $c$ with $|s(c)|=1$.
 The Khovanov complex is a bigraded abelian group $K(D)$ defined below, together with a convenient boundary operator.
 \begin{equation}K(D)=\bigoplus_s \ V(D_s)\{-\sum_c(\mathrm{sign}(c)+s(c))\}\otimes \wedge^{d_s}\Delta_s
 \end{equation}
The cohomological degree $s(D)=\sum_c s(c)$ will be called the height, and the graded degree, equal to the one in the TQFT functor $\V$, up to a shift,
 will simply be called the degree.  
 The shift from the TQFT degree is prescribed by the integer between braces in such a way that
$$\text{q-dim}(G\{i\})=q^i\text{ q-dim}(G)\ .$$
It is convenient to give a local description of the complex. Here we implicitely extend the definition of $K$
to trivalent diagrams, where only $1$-labelled edges are allowed for crossings.
\begin{center}
\vspace{5pt}
\raisebox{5mm}{$K($}
\includegraphics[height=12mm]{G/c+}\raisebox{5mm}{$)\ =\  K($}\includegraphics[height=12mm]{G/c0}\raisebox{5mm}{$)\left\{-1\right\}\ \oplus\  K($}\includegraphics[height=12mm]{G/arete_double.pdf}\raisebox{5mm}{$)\left\{-2\right\} $}\\
\vspace{5pt}
\raisebox{5mm}{$K($}
\includegraphics[height=12mm]{G/c-}\raisebox{5mm}{$)\ =\  K($}\includegraphics[height=12mm]{G/arete_double.pdf}\raisebox{5mm}{$)\left\{2\right\}\ \oplus\  K($}\includegraphics[height=12mm]{G/c0}\raisebox{5mm}{$)\left\{1\right\} $}\\
\end{center}

The boundary operator between summands indexed by states $s$ and $s'$ is zero unless $s$ and $s'$
are different only in one crossing $c$, where $s'(c)=s(c)+1$.
 For a positive crossing (resp. a negative crossing) it is then defined using the TQFT map
associated with the cobordism $\Sigma$, (resp. $\Sigma'$) which are identity outside a neighbourhood of the crossing,
and are given by a saddle with $2$-labelled membrane, around the crossing $c$ with  $s(c)=0$, $s'(c)=1$ (resp.    $s(c)=-1$, $s'(c)=0$).

\centerline{\hspace{-1cm}$\Sigma :$\raisebox{-1.5cm}{\includegraphics[height=38mm]{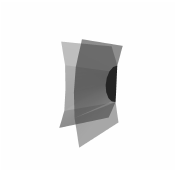}}\ \ \ \ $\Sigma' :$
\raisebox{-1.3cm}{\includegraphics[height=30mm]{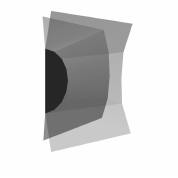}}
}
For a positive crossing $c$:
$$\delta=\V(\Sigma)\otimes (\bullet \wedge c): \V(D_s)\otimes\wedge^{d_s}\Delta_s \rightarrow  
\V(D_{s'})\otimes\wedge^{d_{s'}}\Delta_{s'}$$

For a negative crossing $c$,
$$\delta=\V(\Sigma')\otimes <\bullet,c>: \V(D_{s})\otimes\wedge^{d_{s}}\Delta_{s} \rightarrow  \V(D_{s'})\otimes\wedge^{d_{s'}}\Delta_{s'}$$
Here $<\bullet,c>$ is (the antisymmetrization of) the contraction (using the standard scalar product we understand $c$ as a form).

\begin{thm}
a) $(K(D),\delta)$ is a graded complex.\\
b) If the diagrams $D$ and $D'$ are related by a Reidemeister move, then there exists an graded homotopy equivalence between the complexes $K(D)$ and $K(D')$.\\
c) The graded Euler characteristic is equal to the quantum invariant $\langle D \rangle$, i.e
to $q+q^{-1}$ times the Jones polynomial with change of variable $q=-t^{-\frac{1}{2}}$.
\end{thm}
We will use the notation $Kh(D)$ for the homology of the complex $K$. This theorem says that the
isomorphism class of the graded group $Kh(D)$ is an invariant of the isotopy class of the corresponding link.
 We will see later that for a fixed link $L$ (not considered up to isotopy), then $Kh(L)$
is well defined.
\begin{proof}
We first prove a).
 The map $\partial$ increases the height by one. The elementary cobordism given by a saddle has Euler caracteristic $-1$. The corresponding TQFT map has degree $+1$, and the map $\partial$ on the shifted
TQFT groups has degree $0$. We want now to compute $\partial\circ\partial$.
 The possibly non trivial contributions comes from squares corresponding to states $s$ and $s''$ identical
on all crossings excepted $c_1$ and $c_2$, where $s''(c)=s(c_1)+1$ and $s''(c_2)=s(c_2)+1$.
 We have two intermediate states $s'_1$ ($s'_1(c_1)=s(c_1)+1$ and $s'_1(c_2)=s(c_2)$)    and $s'_2$
 ($s'_2(c_1)=s(c_1)$ and $s'_2(c_2)=s(c_2)+1$) given two contributions represented by the same cobordism $\Sigma$
with two saddles (for the TQFT maps, squares commute). Each of them is twisted. We have two check that 
after twisting the two contributions vanish (squares anticommute) in all cases.\\
If $c_1$ and $c_2$ are both positive crossings, then we get
$$\V(\Sigma)\otimes \left(\bullet\wedge c_1\wedge c_2 +\bullet\wedge c_2\wedge c_1 \right)=0\ .$$
If $c_1$ and $c_2$ are both negative crossings, then we get
$$\V(\Sigma)\otimes \left(<\bullet, c_1\wedge c_2>+<\bullet,c_2\wedge c_1 >\right)=0\ .$$
If $c_1$ is a positive crossing and $c_2$ is a negative crossing, then we get
$$\V(\Sigma)\otimes \left(<\bullet\wedge c_1, c_2>+<\bullet\wedge c_1>\wedge c_2 \right)=0\ .$$

The graded Euler characteristic of the complex $K(D)$ satisfies the Jones skein relation (\ref{skeinJones}), and
is equal to $q+q^{-1}$ for the trivial diagram. Statement c) follows. In the next subsections we will construct homotopy equivalences for each Reidemeister move, and obtain b).
\end{proof}

\subsection{Reidemeister move I}
We first consider the case of a positive crossing.
$$K(\png{3}{-8}{25}{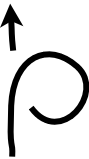})=
\left[K(\png{3}{-8}{25}{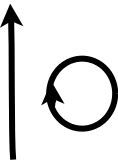})
\stackrel{\delta}{\longrightarrow} 
K(\png{3}{-8}{25}{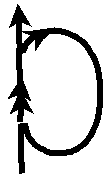})
\right]
$$
Recall that the map $\delta$ is equal to $\V(\Sigma_\delta)\otimes( \bullet \wedge c)$ where $\Sigma_\delta$ contains a saddle with membrane as depicted in
figure \ref{sad0}.
\begin{figure}
 \centerline{\Png{2.5cm}{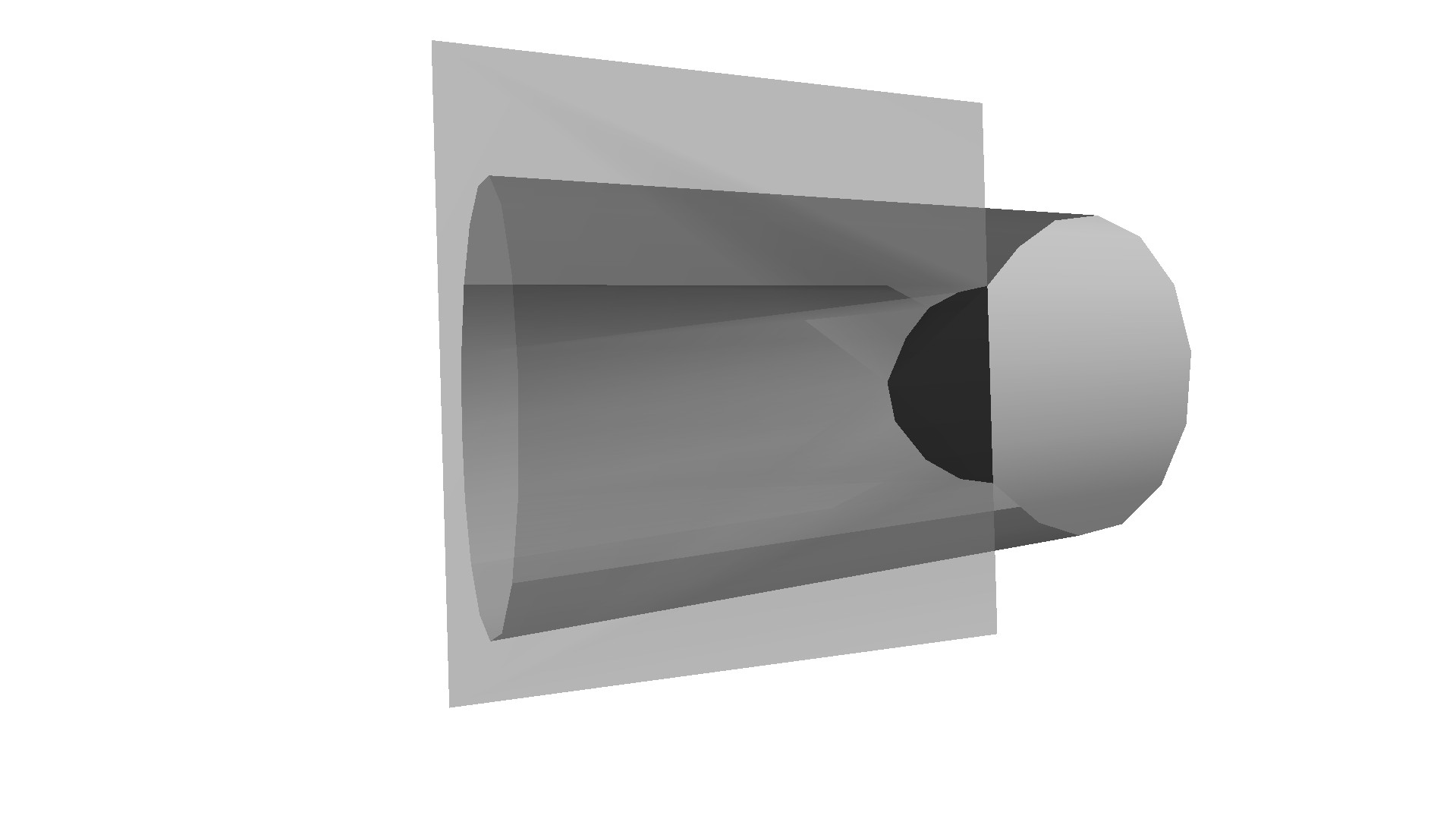}}
 \caption{\label{sad0}}
 \end{figure}
Consider the followings maps.
 $$f:K(\png{3}{-8}{25}{G/curl0})\rightarrow K(\png{3}{-8}{25}{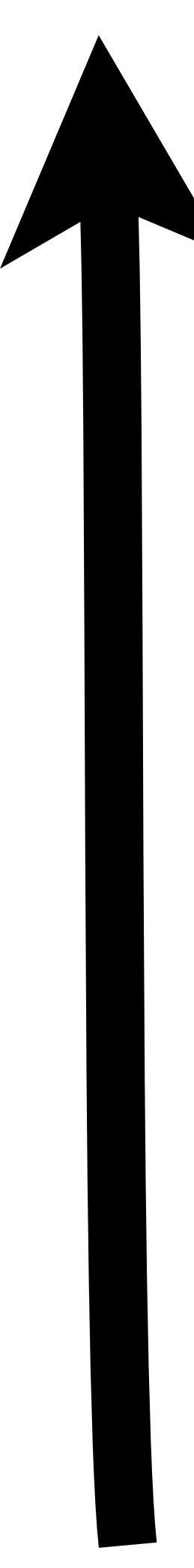})$$ is the TQFT map associated with 
 the
 cobordism in figure \ref{cob_f}.
 \begin{figure}
% \centerline{\Png{2.5cm}{G/cob_f.pdf}}
\centerline{\Png{2.5cm}{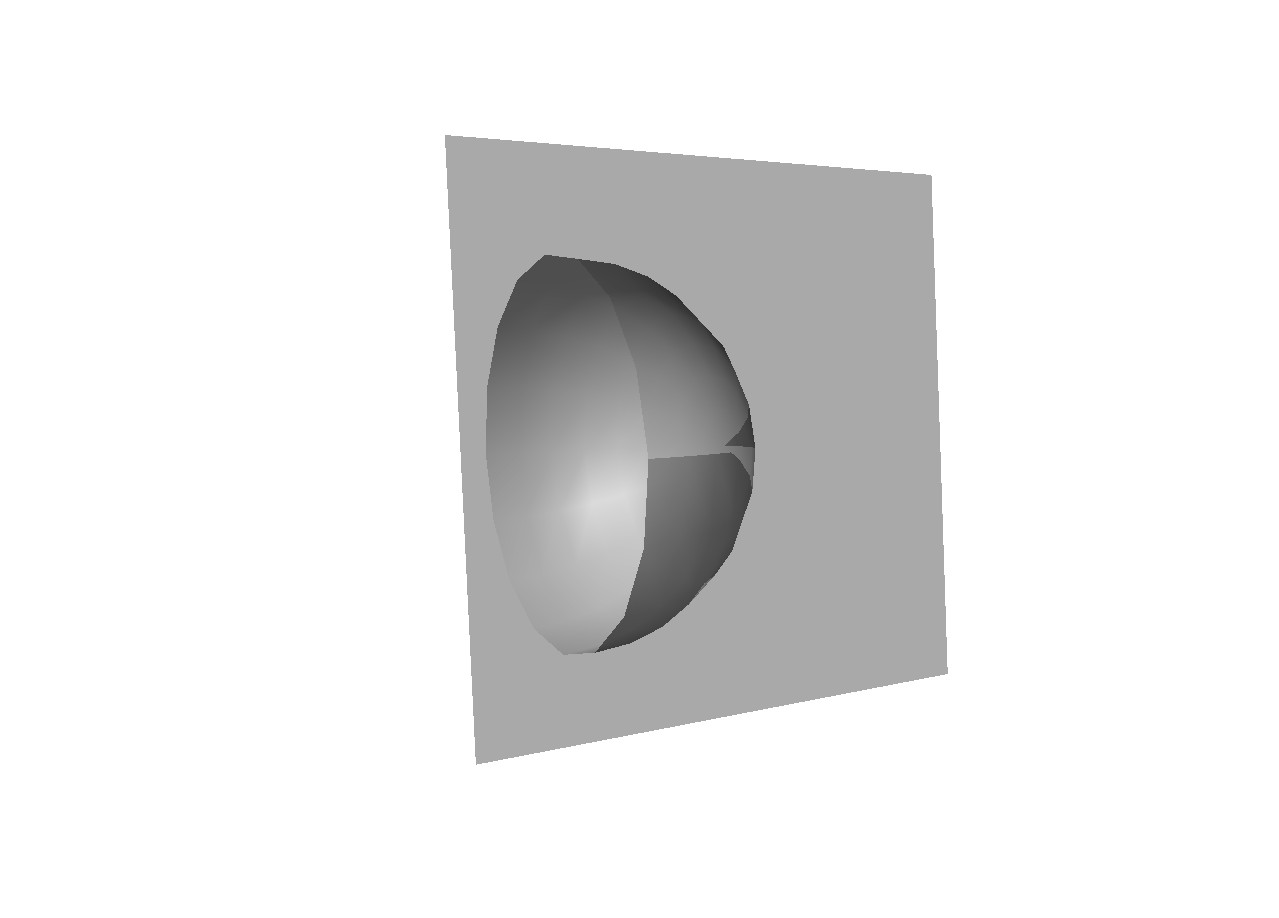}}
 \caption{\label{cob_f}}
 \end{figure}

$$g:K(\png{3}{-8}{25}{G/un})\rightarrow K(\png{3}{-8}{25}{G/curl0})  $$ is the sum of the TQFT maps associated with the
 cobordisms in figure \ref{cob_g}.
% Here the $X$ represents a point on the surface.
 \begin{figure}
 \centerline{\hfill\Png{2.5cm}{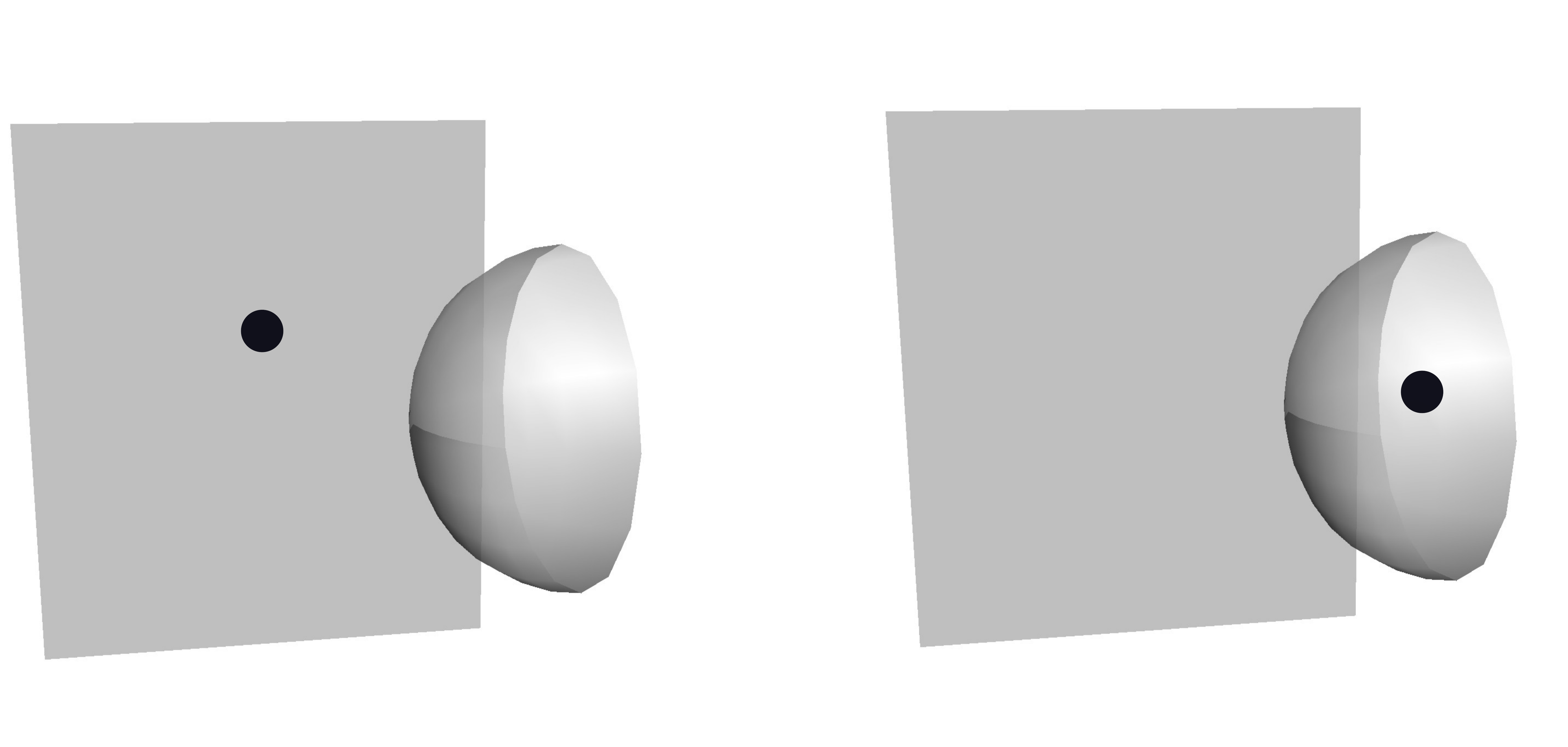}
%\Png{2.5cm}{G/r1b2.pdf}\hfill
}
 \caption{\label{cob_g}}
 \end{figure}
 $$D: K(\png{3}{-8}{25}{G/curl1}\rightarrow K(\png{3}{-8}{25}{G/curl0})$$
is equal to $-\V(\Sigma_D)\otimes <\bullet,c> $ where $\Sigma_D$ is  depicted in
figure \ref{D}.
\begin{figure}
 \centerline{\Png{3cm}{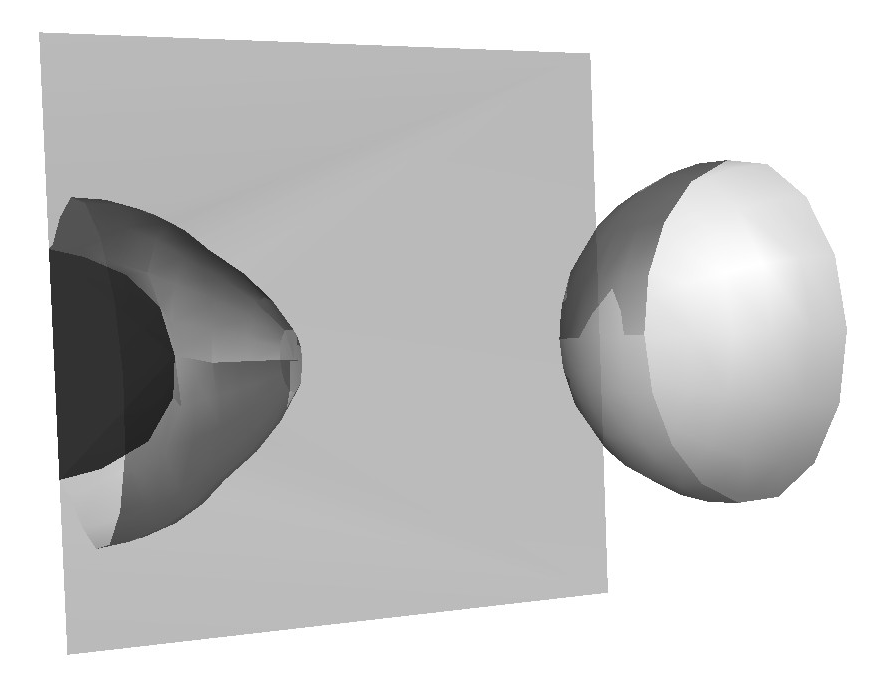}}
 \caption{\label{D}}
 \end{figure}

We have that
$$\delta\circ g=0\ .$$
Hence $f$ and $g$ define chain maps.
We have that
$$f\circ g=Id$$
$$Id-g\circ f=D\circ\delta \ \text{ and }\ \delta\circ D=Id\ .$$
Hence we have that $D$ is an homotopy between $Id$ and $g\circ f$.\\[10pt]

We consider now the case of a negative crossing.
$$K(\png{3}{-8}{25}{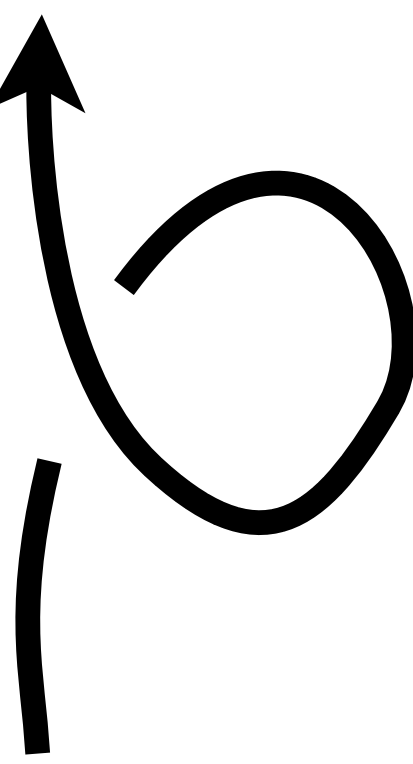})=
\left[K(\png{3}{-8}{25}{G/curl1})
\stackrel{\delta}{\longrightarrow} 
K(\png{3}{-8}{25}{G/curl0})
\right]
$$
Here the map $\delta$ is equal to $\V(\Sigma_\delta)\otimes <\bullet , c>$ where $\Sigma_\delta$ is a saddle. 
Consider the followings maps.
 $$f:K(\png{3}{-8}{25}{G/curl0})\rightarrow K(\png{3}{-8}{25}{G/un})$$ is the TQFT map associated with 
 the
 cobordism in figure \ref{cob_f}.

$$g:K(\png{3}{-8}{25}{G/un})\rightarrow K(\png{3}{-8}{25}{G/curl0})  $$ is the sum of the TQFT maps associated the
 cobordisms in figure \ref{cob_g}. 
$$D: K(\png{3}{-8}{25}{G/curl1}\rightarrow K(\png{3}{-8}{25}{G/curl0})\ .$$
is equal to $\V(\Sigma_D)\otimes <\bullet,c> $ where $\Sigma_D$ is  depicted in
figure \ref{D1}
\begin{figure}
 \centerline{\Png{3cm}{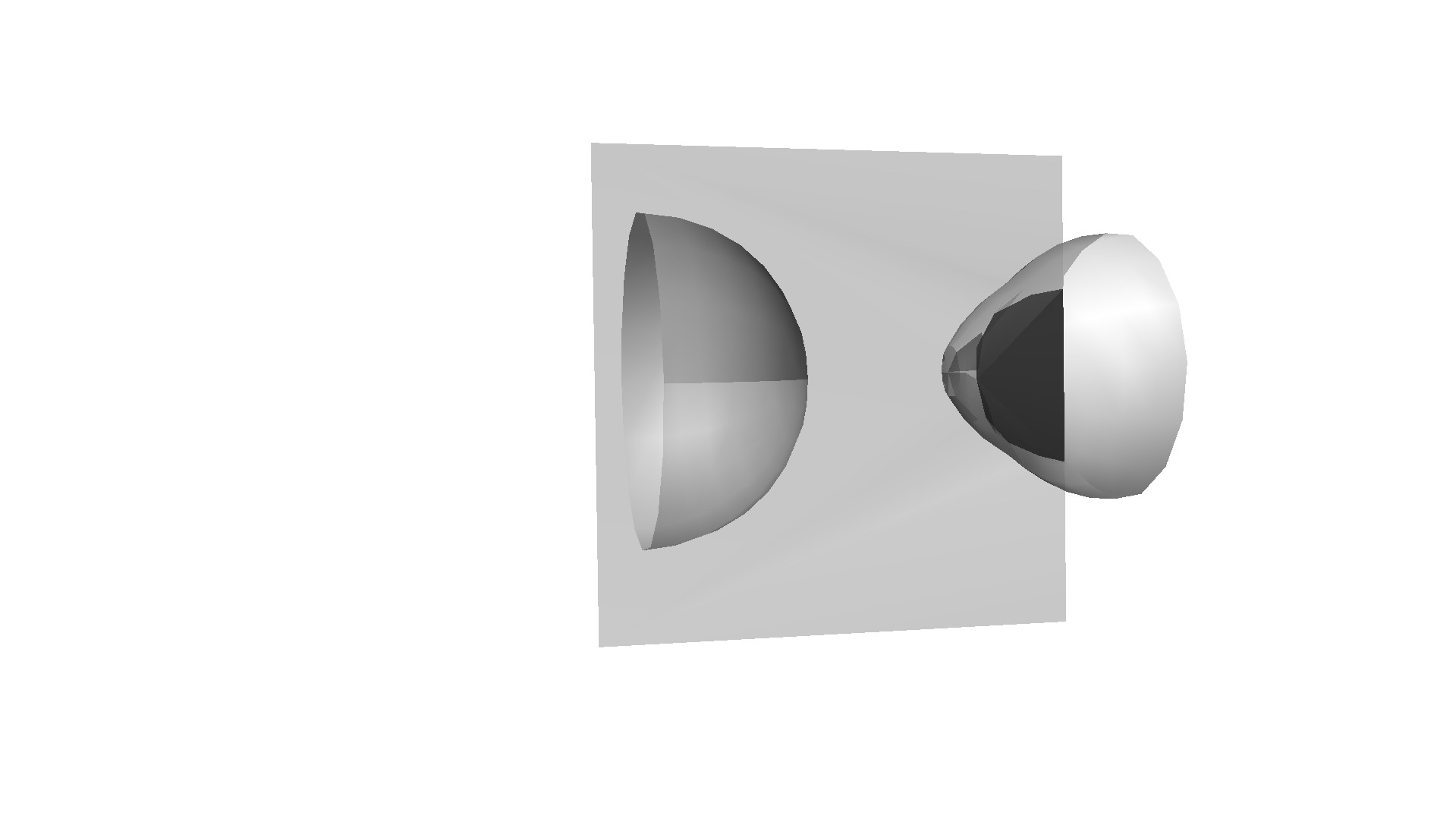}}
 \caption{\label{D1}}
 \end{figure}

We have that
$$f\circ\delta=0\ .$$
Hence $f$ and $g$ define chain maps.
We have that
$$f\circ g=Id$$
$$Id-g\circ f=\delta \circ D\ \text{ and }\ D\circ\delta=Id$$
Hence we have that $D$ is an homotopy between $Id$ and $g\circ f$.
\subsection{Reidemeister II${}_+$}
The complexes we want to consider are described in the diagram below.
$$K\left(\raisebox{-8mm}{\includegraphics[height=20mm]{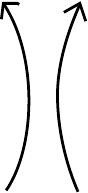}}\right)\ \ ,\ \ 
K\left(\raisebox{-8mm}{\includegraphics[height=20mm]{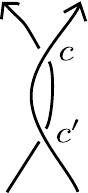}}\right)
=\left[
\raisebox{-23mm}{\includegraphics[height=50mm]{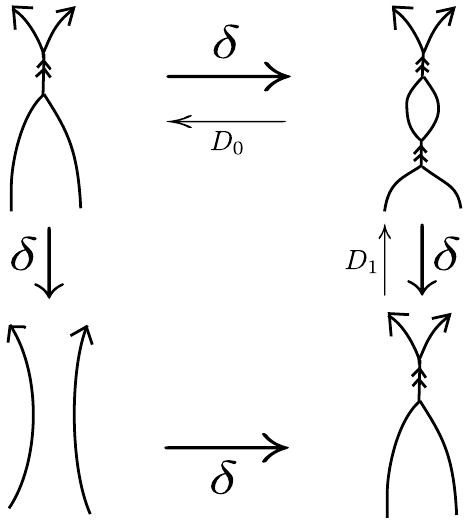}}\right]
$$
Inverse homotopy equivalences are given by the maps $f$ and $g$ defined below
$$f=\one_{\raisebox{-3mm}{\includegraphics[height=6mm]{G/r2+l.pdf}}}\ \oplus\  \V(Z)\otimes (\bullet\wedge c\wedge c')$$
$$g=\one_{\raisebox{-3mm}{\includegraphics[height=6mm]{G/r2+l.pdf}}}\ \oplus\  \V(Z')\otimes <\bullet,c\wedge c'>$$
Here $Z$ and $Z'$ are the trivalent surfaces depicted in figure \ref{ZZZ}
One can check that $\delta\circ f=0$ and $g\circ \delta=0$  ,
so that $f$ and $g$ define chain maps. We have $f\circ g= Id$, moreover there exists $D_0$, $D_1$ as depicted in the diagram above such that 
$$\delta \circ D_1=Id\ ,\ D_0\circ \delta=Id, \ g\circ f+D_1\circ \delta+\delta \circ D_0=Id\ .$$
\begin{exo}
 Find $D_0$, $D_1$ as expected (hint: use the $4$ terms relation in ...).
\begin{figure}
 \centerline{\raisebox{1.4cm}{$Z:$}\hspace{-4cm}\Png{3cm}{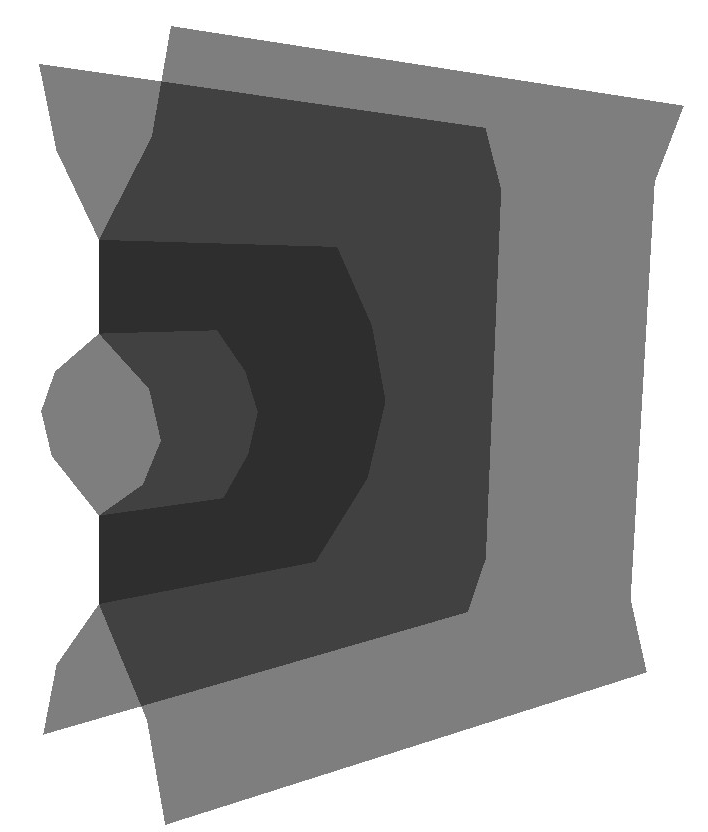}\hspace{-2cm}\raisebox{1.4cm}{$Z':$}\hspace{-5cm} \Png{3cm}{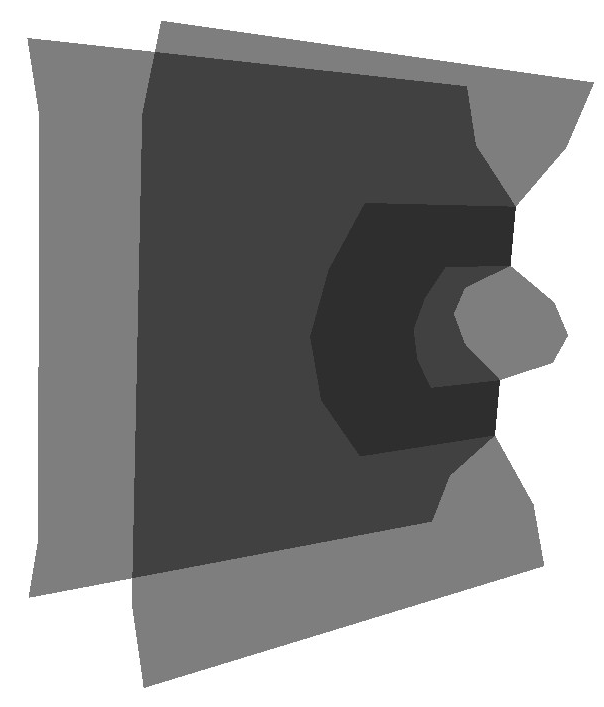}\hspace{-4cm}}
 \caption{\label{ZZZ}}
 \end{figure}

\end{exo}
\subsection{Reidemeister II${}_-$}
Homotopy equivalences for negative Reidemeister move are defined in a similar way. The homotopy equivalence checking rests essentially on the relation in Lemma \ref{squares}a.
\subsection{Reidemeister III}
We have to consider the complex decomposed as a cube as described below, and the symmetric one.
$$
K\left(\raisebox{-8mm}{\includegraphics[height=20mm]{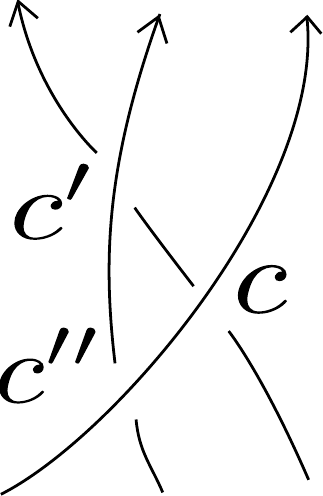}}\right)
=\left[
\raisebox{-28mm}{\includegraphics[height=60mm]{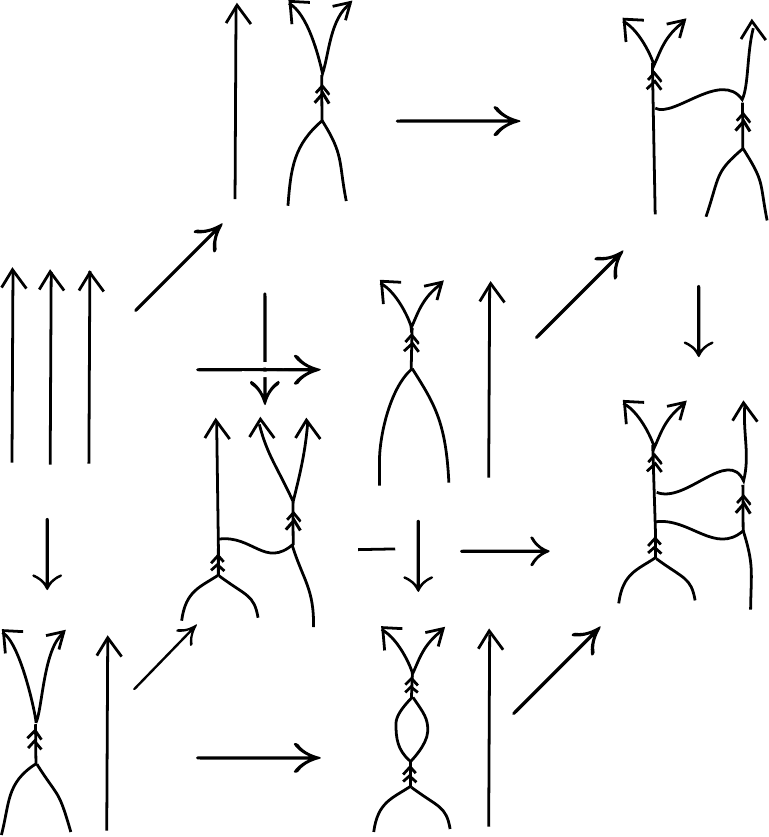}}
\right]
$$
We can find an acyclic subcomplex.
\begin{lem}
 The subcomplex described below is acyclic.
$$\left[
\raisebox{-8mm}{\includegraphics[height=20mm]{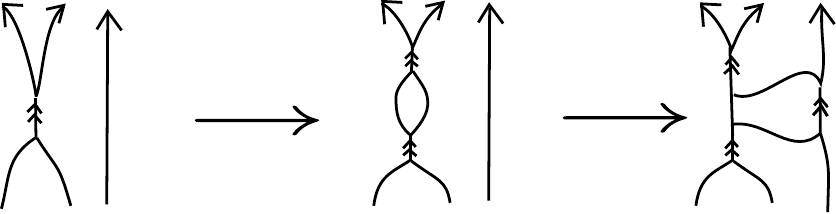}}
\right]
$$
\end{lem}
We then apply a Gauss reduction (see e.g. ...) and obtain
an homotopy equivalence with a smaller comlex. Note that the above acyclic subcomplex is not a direct summand so that we have to
carefully recalculate the boundary map including the action on the twisting determinant. The result is described below
Here the boundary maps are given by a saddle with
$2$ labelled membrane as before twisted as indicated near the arrows.
$$
K\left(\raisebox{-8mm}{\includegraphics[height=20mm]{G/r3l.pdf}}\right)
\simeq\left[
\raisebox{-18mm}{\includegraphics[height=40mm]{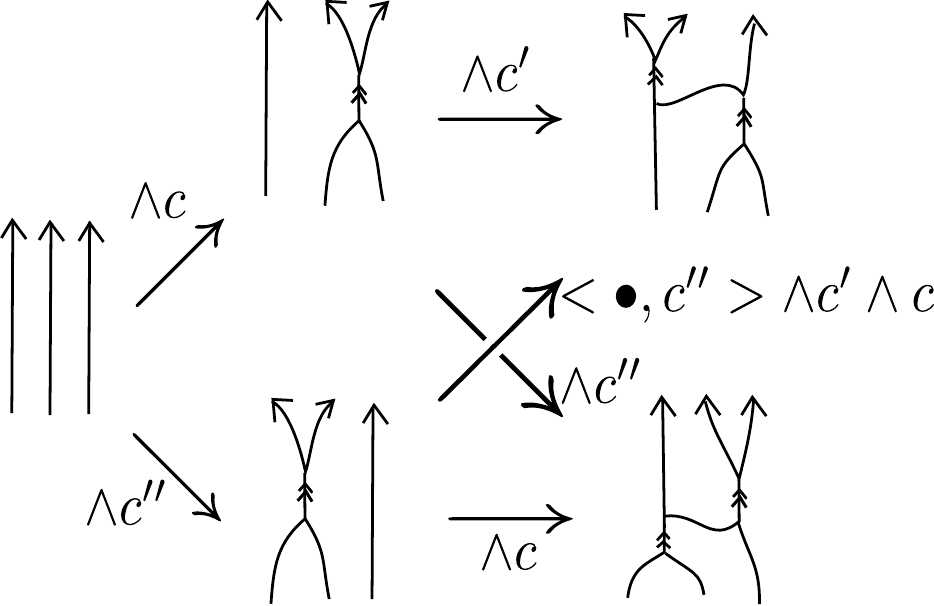}}
\right]
$$
We have 
$$
K\left(\raisebox{-8mm}{\includegraphics[height=20mm]{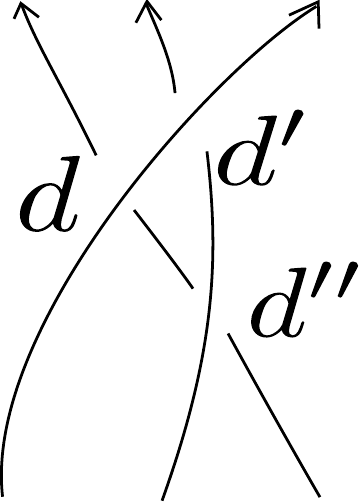}}\right)
\simeq\left[
\raisebox{-18mm}{\includegraphics[height=40mm]{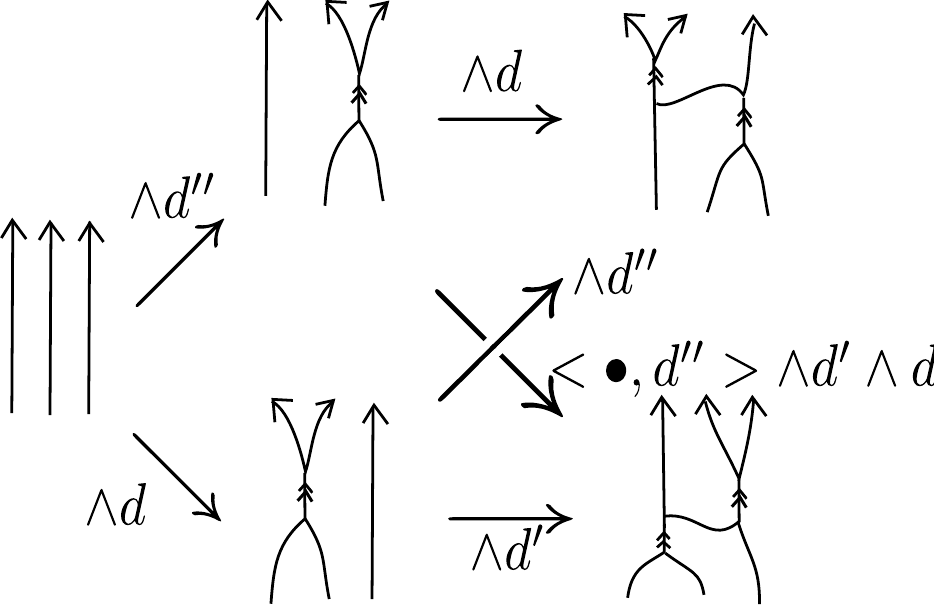}}
\right]
$$
An isomorphism between the two complexes is obtained by using the idendity map on the TQFT modules,
and $ c\mapsto d''\ ,\ c'\mapsto d'\ ,\ c''\mapsto d$ on the twisting determinants.
\section{Lee-Rasmussen spectral sequence}
Following Lee and Rasmussen, we consider now the Frobenius algebra
\mbox{$\mathbf{A'}=\Z[X]/X^2-1$}, and the associated TQFT functor $\V'$.
The preceeding construction applies as well, excepted that here the TQFT functor $\V'$ is no more
graded, but filtered. A generator of a TQFT module $\V'(\gamma)$ is given the same degree
as before, and $F^j(\V'(\gamma))$ is spanned by generators with degree less or equal to $j$.
 Observe that a cobordism $\Sigma$ induces a filtered map, with degree equal to
\mbox{$\chi(\Sigma)-2(\sharp \mathrm{points})$}.

We consider the filtered abelian group $K'(D)$ defined below,
with the same notation as before.
 \begin{equation}K'(D)=\bigoplus_s \V'(D_s)\{-\sum_c(\mathrm{sign}(c)+s(c))\}\otimes \wedge^{d_s}\Delta_s
 \end{equation}
The boundary operator $\partial'$ is still defined by twisting the TQFT map associated with a saddle.  We denote by $Kh'(D)$ the homology of this complex.
\begin{thm}
a) $(K'(D),\partial')$ is a filtered chain complex.\\
b) If the diagrams $D$ and $D'$ are related by a Reidemeister move, then there exists a filtered homotopy equivalence between the complexes $K'(D)$ and $K'(D')$.\\
c) There exists a spectral sequence whose second page is the homology of our complex $K$
$$E_2^{i,j}(D)=Kh^{i-j,j}(D)\ ,$$ which converges to ${Kh'}^*(D)$.
\end{thm}
\begin{proof}
 a) and b) are proved as before. Statement c) follows from standard facts with filtered chain complexes.
\end{proof}

The theorem shows that all the pages with index greater or equal to $2$ are invariants
of the link. For Khovanov original homology, it was proved by Lee \cite{Lee} that the
limit depends only on the number of components. Rasmussen \cite{Ras} was able to extract  a lower bound for the slice genus and to use it to give a combinatorial proof of Milnor conjecture on the slice genus of torus knots.

We will compute our oriented version of Lee-Rasmussen homology over \mbox{$\Lambda=\Z[\frac{1}{2}]$} using
the Karoubi completion method of Bar-Natan and Morrisson \cite{BM}. We denote by $Kh'(D,\Lambda)$ the homology of $K'(D)\otimes \Lambda$.
\begin{thm}
For a link diagram $D$ with $m$ components, $Kh'(D,\Lambda)$ is a free $\Lambda$-module of rank $2^m$,
with a canonical basis indexed by maps $\epsilon : \pi_0(D)\rightarrow \{\pm 1\}$.
\end{thm}

\begin{proof}
 The algebra $\mathbf{A}'$ contains the minimal idempotents
$$\pi_\pm=\frac{\mathbf{1} \pm X}{2}\ .$$
 Using these idempotents we extend the TQFT functor $\V'$ to an extended trivalent category where $1$-labelled
edges or faces may be colored with $\pi_\pm$. If $1$-labelled edges in a trivalent graph $\gamma$ are colored
with a sequence of signs denoted by $\epsilon$, then 
we obtain an object $\gamma(\epsilon)$ whose associated module is the image of the obvious  projector
$\pi_\epsilon\in \V'(\gamma)$, associated with $\epsilon$.
 The relation in Lemma \ref{lemabcd} a) still hold for the functor $\V'$.
We deduce that the module $\V'(\gamma)$ is zero if signs agree on the two $1$-labelled edges adjacent
to a vertex.
 For an {\em alternating} sign assignement $\epsilon$ , the  module $\V'(\gamma(\epsilon))$ has rank one with a basis represented by  a trivalent surface whith only discs (without points) as $2$-cells.
 The complex $K'(D)$ which was a sum indexed by states, is now decomposed into a sum indexed
by colored states (states with coloring of all arcs). The boundary map can be computed locally.
 The TQFT map associated to a saddle is zero unless all colors coincide on the arcs belonging to the same $1$-labelled component.
 In the remaining case, this TQFT map is an isomorphism. We obtain a deformation retract on a subcomplex $K''(D)$ where the boudary map is zero.
Locally, i.e. for a crossing, the subcomplex $K''(D)$ is described below.
$$ DESSIN$$
For a generator, the signs associated with arcs belonging to the same  component
of the represented link are the same. Moreover for an assignment of signs on the components,
the state of each crossing is determined so there
is a unique corresponding generator.
\end{proof}

\section{Functoriality}
Extension of Khovanov homology to link cobordisms and functoriality up to sign
was conjectured by Khovanov and established by Jacobsson \cite{Jac} and Khovanov \cite{Kh3}, and  Bar-Natan \cite{BN2}. The sign ambiguity was carried over by Clark-Morrison-Walker \cite{CMW}. In this section we show that our construction has a strictly functorial extension.

A  movie description of a cobordism is a generic projection in $[0,1]\times \R^2$ of a smooth surface
in $[0,1]\times \R^3$. Generically a movie decomposes into elementary ones
 which either
describe a Reidemeister move or glue an  handle to the surface.
 To each Reidemeister type movie we associate the corresponding homotopy equivalence,
and to each handle addition we associate the chain map associated with the corresponding TQFT map.
 We then associate to a movie the composition of these elementary chain maps.
\begin{thm}
a) Let $L\subset S^3$ be a link represented by diagrams $D$ and $D'$, then the homology isomorphism
associated to a sequence of Reidemeister moves is canonical, i. e. $Kh(L)$ is well defined.\\
b) Let $C\subset [0,1]\times \S^3$ be a smooth cobordism between the links $L$ and $L'$. The homology map
$Kh(L)\rightarrow Kh(L')$ induced by a movie description of $C$ only depends of the isotopy class of $C$
rel. $L\times\{0\}\cup L'\times\{1\}$, and $Kh$ extends to a functor on the embedded cobordism category.
\end{thm}
\begin{rem}
 Note that here we consider fixed links, and not links up to isotopy. It was observed by Jacobson that Khovanov homology
do have monodromy.
\end{rem}
\begin{proof}
 In a) we may consider the link in $\R^3$ as well. Then a diagram is associated with a generic projection on an oriented plane,
and is parametrized by a point in $S^2$. A generic loop in $S^2$ result in a finite sequence of Reidemeister
moves, and induces an homology isomorphism. An homotopy to the trivial loop can be obtained via elementary homotopies
corresponding to Reidemeister type movie moves. So statement a) results from Lemma \ref{moviemoves}.
 More generally movie moves generate isotopies of embedded surfaces whence staement b).
\end{proof}

\begin{lem}\label{moviemoves}
For each movie move the two induced homology maps coincide.
\end{lem}

\begin{proof}
 We first show projective functoriality, using a simplicity argument
of Bar-Natan. Then we consider the degenerate Lee theory. For a given
movie move, the sign is the same for the degenerate theory. We have seen that $Kh'(L)$ has a distinguished
element associated with the assignement of a positive sign on all components.
 We can see that the map associated with an elementary movie respects this canonical element.
\end{proof}

\end{document}